\documentclass[sn-mathphys-num, lin-no]{sn-jnl}

\usepackage{graphicx}%
\usepackage{multirow}%
\usepackage{amsmath,amssymb,amsfonts}%
\usepackage{amsthm}%
\usepackage{mathrsfs}%
\usepackage[title]{appendix}%
\usepackage{xcolor}%
\usepackage{textcomp}%
\usepackage{manyfoot}%
\usepackage{booktabs}%
\usepackage{algorithm}%
\usepackage{algorithmicx}%
\usepackage{algpseudocode}%
\usepackage{listings}%

\theoremstyle{thmstyleone}%
\theoremstyle{thmstyletwo}%
\newtheorem{remark}{Remark}%
\newtheorem{lemma}{Lemma}%
\newtheorem{corollary}{Corollary}%

\theoremstyle{thmstylethree}%

\begin{document}

\title[Article Title]{A Deterministic Algorithm of Quasi-Polynomial Complexity for Clipped Cubes Volume Approximation}


\author{\fnm{Marius} \sur{Costandin}}\email{costandinmarius@gmail.com}
%
%
%
%
%


\abstract{We give a deterministic method of quasi-polynomial complexity to approximate the volume of the intersection of the unit hypercube with two specific sets. The method can actually be applied (without losing the quasi-polynomial complexity) to compute the volume of the hypercube intersected with a fixed number of sets, described by equations of the form $\sum_{q=1}^n a_q(x_q) \leq b$, where $a_q : \mathbb{R} \to \mathbb{R}$ are polynomial functions and $b \in \mathbb{R}$. Note that the resulting sets are not necessarily convex. This type of equations describe, among others, half-spaces, balls and ellipsoids. We give detailed convergence and complexity analysis for the case in which the unit hypercube is clipped by balls of arbitrary radius but with centers whom distance to the unit hypercube is greater than $1$ (one). }




\maketitle
\tableofcontents

\section{Introduction} \label{intro}

In this paper we study the computation of volumes in higher dimensions. In particular we are concerned with the approximation of the volume of the unit hypercube intersected with one or two balls or intersected with a ball and a half space. We give a method whose complexity class is the same for balls, half-spaces or any other set which ca be described by the equation $\sum_{k=1}^n a_k(x_k) \leq b$ for some $b \in \mathbb{R}$ and $a_k(\cdot) $ are real univariate polynomials. 

Note that being able to compute volumes of such sets allows one to solve difficult optimization problems. For instance consider the problem
\begin{align}
\max_{x \in \mathcal{U} \cap H} \|x - C_0\|^2
\end{align}  where $\mathcal{U}$ is the unit hypercube and $\mathcal{H}$ is a half-space. According to \cite{sahni_comp}, this problem is NP-Complete, since the Subset Sum problem can be reduced to it. However, it can be easily solved using a volume computation oracle as follows:

\begin{align}\label{E2f}
\max_{x \in \mathcal{U} \cap H} \|x - C_0\| = \min \{R > 0 | \text{vol}(\mathcal{U} \cap \mathcal{H} \cap \mathcal{B}(C_0,R)) \geq \text{vol} (\mathcal{U} \cap \mathcal{H}) \} 
\end{align} Indeed, the radius of the smallest ball centered in $C_0$ which includes $\mathcal{U} \cap \mathcal{H}$ is the largest distance from a point in $\mathcal{U} \cap \mathcal{H}$ to $C_0$. One can solve (\ref{E2f}) to $\epsilon$ precision with at most $\mathcal{O}\left(\log\frac{1}{\epsilon} \right)$ calls to such a volume computation oracle.

Deterministic algorithms in literature for computing volumes of clipped hypercubes are mainly concerned with the intersection of a hypercube with a single half-space or a fixed number of half-spaces, see \cite{clip_hchs}, \cite{clip_hchs1}, \cite{jim_law} and the references therein. 

We also mention the celebrated result in \cite{cvx_rdm} which gives a randomized polynomial time approximation scheme for the computation of the volume of a convex body. 

Throughout the paper we denote by $\mathcal{B}(C_0,r)$ the ball centered in $C_0 \in \mathbb{R}^n$ of radius $r$. We denote by $1_{n \times 1}$ the element of $\mathbb{R}^n$ whose entries are all $1$. We use $\text{poly}(n)$ in complexity classes to refer to a polynomial in $n$. Finally, by $d(x, \mathcal{U})$ we denote the distance from the point $x$ to the convex set $\mathcal{U}$.   

The outline of the paper is the following:

\begin{enumerate}
\item In Subsection \ref{AH} we give an original approximation of the Heaviside step function in terms of an "incomplete" Gauss integral. Our approximation is useful when computing volumes because by changing the  integration order, the resulting integrand can be approximated using globally convergent McLaurin series.
\item In Subsection \ref{VAE} we show that it is sufficient to retain only a polynomial number of elements in a McLaurin series to attain a desired accuracy of a specific approximation. We prove this by showing that the series reminder in Lagrange form vanishes in certain conditions.
\item In Subsection \ref{EI1} and \ref{EI2} we give an original recurrent method which at each step obtains an $n-$integral as a sum of a polynomial number of products of $\frac{n}{2}-$integrals. This allows a quasi-polynomial overall complexity, since we end up with univariate integrals over some polynomial functions, which can be computed exactly.
\item In Subsection \ref{ErAn} we give an original upper bound on the quadrature error using Riemann sums and the magnitude of the derivative of the function to be integrated.  
\end{enumerate}

\section{Main Results}

Let $n \in \mathbb{N}$, $a_i (\cdot), c_i(\cdot) $ be real univariate polynomials for all $i \in \{1, \hdots, n\}$ and $b,d \in \mathbb{R}$ fixed. Define the sets:
\begin{align}
\mathcal{S}_1 = \left\{x \in \mathbb{R}^n | \sum_{i=1}^n a_i(x_i) \leq b \right\} \hspace{0.5cm} \mathcal{S}_2 = \left\{x \in \mathbb{R}^n | \sum_{i=1}^n c_i(x_i) \leq d \right\}
\end{align} and the unit hypercube $\mathcal{U} = \{ x \in \mathbb{R}^n | 0 \leq x_i \leq 1, \forall i \in \{1, \hdots, n\}\}$. It is clear that
\begin{align}
\max_{x \in \mathcal{U}} \left| \sum_{i=1}^n a_i(x_i) - b \right| \in \mathcal{O}(\text{poly}(n))  \hspace{0.5cm} \max_{x \in \mathcal{U}} \left| \sum_{i=1}^n c_i(x_i) - d \right| \in \mathcal{O}(\text{poly}(n))  
\end{align}

Let $\rho_0: \mathbb{R}^n \to \mathbb{R} $ and $\rho_1: \mathbb{R}^n \to \mathbb{R} $ with 
\begin{align}
\rho_0(x) = b - \sum_{i = 1}^n a_i(x_i) \hspace{0.5cm} \rho_1(x) = d - \sum_{i=1}^n c_i(x_i)
\end{align}

In the following we are concerned with computing the volume of:
\begin{align}
\mathcal{L} = \mathcal{U} \cap \mathcal{S}_1 \cap \mathcal{S}_2
\end{align}

We start wih: 
\begin{align}\label{E9}
\text{vol}(\mathcal{L}) & =\int_{-\infty}^{\infty} dx_1 \hdots \int_{-\infty}^{\infty} dx_n \cdot  H(x_1) \cdot H(1-x_1)\cdot \hdots \cdot H(x_n) \cdot \nonumber \\
& \cdot H(1-x_n) \cdot  H(\rho_0(x)) \cdot H(\rho_1(x)).
\end{align} 

\subsection{Approximation of the Heaviside step function} \label{AH}

In (\ref{E9}) we use the Heaviside step function $H(x)$ to form the indicator function of $\mathcal{L}$. A well known parameterized approximation of this function is the logistic sigmoid function
\begin{align}\label{E10}
H(t) \approx L_K(t) := \frac{e^{K \cdot t}}{1 + e^{K \cdot t}}  \hspace{0.5cm} \forall t\in \mathbb{R}
\end{align} for any $ K > 0$, a parameter which controls the sharpness of the approximation.  Analyzing the sigmoid function $L_K(t)$ one obtains:
\begin{align}
\frac{d}{dK} L_K(t) &= \frac{e^{K \cdot t}}{1 + e^{K \cdot t}} \cdot t - \left( \frac{e^{K\cdot t}}{1 + e^{K\cdot t}}\right)^2\cdot t \nonumber \\
& = t \cdot \frac{e^{K\cdot t}}{1 + e^{K\cdot t}} \cdot \left( 1 - \frac{e^{K\cdot t}}{1 + e^{K\cdot t}} \right) \nonumber \\
& = t \cdot \frac{1}{1 + e^{-K \cdot t}} \cdot \frac{1}{ 1 + e^{K \cdot t}} = t \cdot \Phi_L(K \cdot t)
\end{align} where 
\begin{align}
\Phi_L(t) = \frac{1}{2 + e^{t} + e^{-t}} 
\end{align} One can see in Figure \ref{Fig1} the plot of this function

\begin{figure}[h]
\centering
\includegraphics[scale = 0.5]{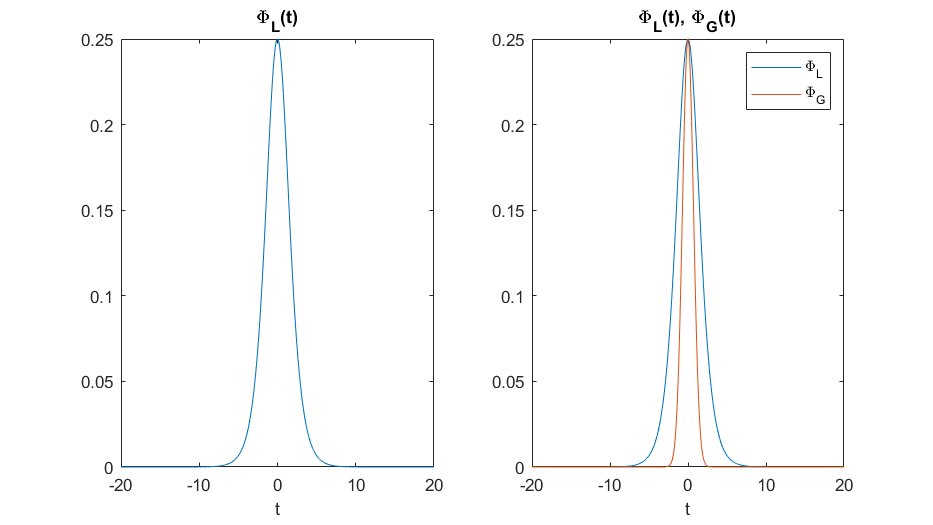}
\caption{Plot of a) $\Phi_L(t)$ b) $\Phi_L(t)$ and $\Phi_G(t)$}
\label{Fig1}
\end{figure}

One naturally asks whether it is possible to obtain a step function approximation if this function $\Phi_L(\cdot)$ is replaced by a similar shaped function, say $\Phi_G(\cdot)$. We choose $\Phi_G(t) = \frac{1}{4}\cdot e^{-t^2}$ and obtain a step function approximation as follows:
\begin{align}\label{E13}
H_K(t) = \frac{1}{2} + \frac{1}{\sqrt{\pi}} \cdot  \int_{0}^K t \cdot e^{-(t \cdot y)^2} \cdot dy
\end{align} Indeed, $H_K(t)$ is a parameterized approximation of the step function since $\int_{-\infty}^{\infty} e^{-(t\cdot y)^2} \cdot dy = \frac{\sqrt{\pi}}{|t|}$ hence
\begin{align}\label{E14}
\lim_{K \to \infty} H_K(t) = \frac{1}{2} + \frac{1}{2} \cdot \frac{t}{\sqrt{\pi}} \cdot \frac{\sqrt{\pi}}{|t|} = \frac{1 + \frac{t}{|t|}}{2} = \begin{cases} 0, \hspace{0.2cm} &t<0 \\ \frac{1}{2}, & t = 0 \\ 1, & t > 0\end{cases}
\end{align}

See Figure \ref{Fig2} for a plot of the function $H_K(t)$.

\begin{figure}[h]
\centering
\includegraphics[scale = 0.5]{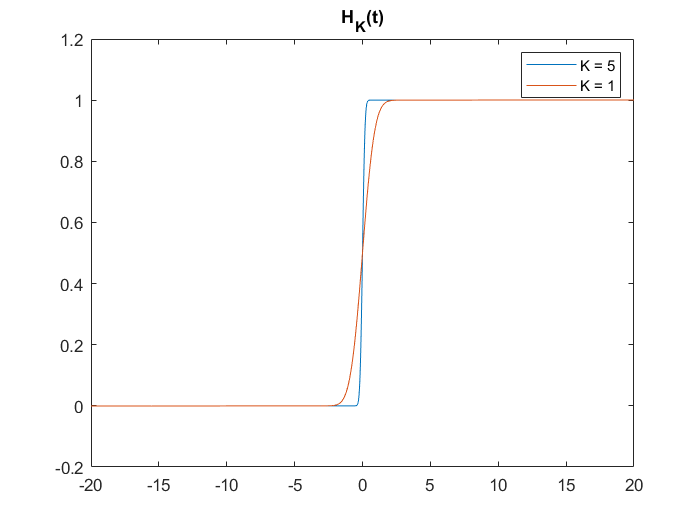}
\caption{Plot of $H_K(t)$ for various values of $K$}
\label{Fig2}
\end{figure}

From (\ref{E9}), (\ref{E13}) and (\ref{E14}) one gets
\begin{align}
\text{vol}(\mathcal{L}) &= \lim_{K \to \infty} \int_{-\infty}^{\infty} dx_1 \hdots \int_{-\infty}^{\infty} dx_n \cdot  H(x_1) \cdot H(1-x_1)\cdot \hdots \cdot H(x_n) \cdot \nonumber \\
& \cdot H(1-x_n) \cdot  H_K(\rho_0(x)) \cdot H_K(\rho_1(x)) \nonumber \\
& = \lim_{K \to \infty} T(K)
\end{align} where 

\begin{align}
T(K) =& \int_{-\infty}^{\infty} dx_1 \hdots \int_{-\infty}^{\infty} dx_n \cdot  H(x_1) \cdot H(1-x_1)\cdot \hdots \cdot H(x_n) \cdot \nonumber \\
& \cdot H(1-x_n) \cdot  H_K(\rho_0(x)) \cdot H_K(\rho_1(x)) 
\end{align}
\subsection{Volume Approximation Error Analysis} \label{VAE}

Because we want to compute $T(K)$ instead of the actual volume, for a fixed $K$, one is interested in the difference $\left| \text{vol}(\mathcal{L}) - T(K) \right| $. We give an upper bound on this and in particular show that in case the hypercube is intersected with balls, the quantity $\left| \text{vol}(\mathcal{L}) - T(K) \right| $ is in $\mathcal{O} \left( \frac{\text{poly}(n) \cdot \sum \left\|C_i - \frac{1}{2} \cdot 1_{n \times 1} \right\|}{K}\right)$. The condition for this to happen, we show, is that the smallest distance to a point in the hypercube from any center of the balls, the point $C_i$, should be greater than $1$ (one), i.e. $d(C_i,\mathcal{U}) \geq 1$. 

\begin{align}\label{E17c}
& \left| \text{vol}(\mathcal{L}) - T(K) \right| \leq \nonumber \\
&  \int_{0}^1 dx_1 \hdots \int_{0}^1 dx_n \cdot \left| H(\rho_0(x)) \cdot H(\rho_1(x)) - H_K(\rho_0(x)) \cdot H_K(\rho_1(x)) \right|. 
\end{align} Since 
\begin{align}
& H(\rho_0) \cdot H(\rho_1) - H_K(\rho_0) \cdot H_K(\rho_1) = \nonumber \\
& =H_K(\rho_0) \cdot \left(H(\rho_1) - H_K(\rho_1) \right) + \left( H(\rho_0) - H_K(\rho_0) \right) \cdot H(\rho_1)
\end{align} and $|H_K(\cdot)|, |H(\cdot)| \in [0,1]$ we get in (\ref{E17c})

\begin{align}\label{E17d}
\left| \text{vol}(\mathcal{L}) - T(K) \right| \leq  &  \int_{0}^1 dx_1 \hdots \int_{0}^1 dx_n \cdot \left| H(\rho_0(x))  - H_K(\rho_0(x)) \right| + \nonumber \\
& +  \int_{0}^1 dx_1 \hdots \int_{0}^1 dx_n \cdot \left| H(\rho_1(x))  - H_K(\rho_1(x)) \right| 
\end{align} We now focus w.l.o.g. on one integral. For this, we analyze the behavior of $H(t) - H_K(t)$.  From (\ref{E14}) it is obtained

\begin{align}\label{18d}
H(t) - H_K(t) = \frac{t}{\sqrt{\pi}}\int_{K}^{\infty} e^{-(t\cdot y)^2} \cdot dy
\end{align} Integrating (\ref{18d}) from $0$ to $\infty$ one obtains for $K > 0$
\begin{align}
\int_{0}^{\infty} \left( H(t) - H_K(t) \right) \cdot dt &= \frac{1}{\sqrt{\pi}} \cdot \int_{0}^{\infty} dt \int_{K}^{\infty} t \cdot e^{-(t \cdot y)^2} \cdot dy \nonumber \\
& = \frac{1}{\sqrt{\pi}} \cdot \int_{K}^{\infty} dy \int_{0}^{\infty} t \cdot e^{-(t \cdot y)^2} \cdot dt \nonumber \\
& = -\frac{1}{\sqrt{\pi}} \cdot \int_{K}^{\infty} dy \cdot \frac{e^{-(t\cdot y)^2}}{2 \cdot y^2} \biggr|_{t = 0}^{\infty} \nonumber \\
& = \frac{1}{\sqrt{\pi}} \cdot \int_{K}^{\infty} \frac{1}{2 \cdot y^2} \cdot dy = \frac{1}{2 \cdot \sqrt{\pi} \cdot K}.
\end{align}

Now, let $f:\mathbb{R} \to \mathbb{R}$ with $|f'(t)| \geq 1$. For the ease of presentation we assume $f \geq 0$ and $f' \geq 0$ and obtain $\sqrt{\pi} \cdot \int_{0}^{\infty} \left| H(f(t)) - H_K(f(t)) \right| \cdot dt$ as
\begin{align}\label{E20c}
\int_{0}^{\infty} dt \cdot f(t) \int_{K}^{\infty} dy\cdot e^{-(f(t)\cdot y)^2} &\leq \int_{0}^{\infty} dt \cdot f'(t) \cdot f(t) \int_{K}^{\infty} dy\cdot e^{-(f(t)\cdot y)^2} \nonumber \\
& \leq - \int_{K}^{\infty} dy\cdot \frac{e^{-(f(t)\cdot y)^2}}{2 \cdot y^2} \biggr|_{t = 0}^{\infty} \nonumber \\
& \leq \int_{K}^{\infty} \frac{1}{2\cdot y^2} = \frac{1}{2\cdot K}
\end{align}

 Let $\rho:\mathbb{R}^n \to \mathbb{R}$ and $f(t) = \rho(c + t \cdot u)$ where $c,u \in \mathbb{R}^n$, $\|u\| = 1$. The condition $|f'(t)| \geq 1$ becomes
\begin{align}
\left|\frac{\partial \rho}{\partial x}(c + u \cdot t) \cdot u \right| \geq 1 
\end{align} 

In the following we give an upper bound on $|\text{vol}(\mathcal{L}) - T(K) | $ for the situation where $\mathcal{L}$ is the intersection of two balls with the unit hypercube.

\begin{lemma} \label{L1}
Let $C_0, C_1 \in \mathbb{R}^n$, $\mathcal{L} = \mathcal{U} \cap \bar{\mathcal{B}}(C_0,r_0) \cap \bar{\mathcal{B}}(C_1, r_1)$ with 
\begin{align}\label{E22f}
\min_{x \in \mathcal{U}} \|x - C_i\| \geq 1 \hspace{0.5cm} \forall i \in \{0,1\}
\end{align} Define  
\begin{align}
\mathcal{K}_0 &= \text{conv}(C_0, \mathcal{U}) := \{\lambda \cdot x + (1- \lambda)\cdot y| \lambda \in [0,1], x,y \in \mathcal{U} \cup \{C_0\}\} \nonumber \\
\mathcal{K}_1 &= \text{conv}(C_1, \mathcal{U}) := \{\lambda \cdot x + (1- \lambda)\cdot y| \lambda \in [0,1], x,y \in \mathcal{U} \cup \{C_1\}\} 
\end{align} then 
\begin{align}
| \text{vol}(\mathcal{L}) - T(K) | \leq \frac{n}{\sqrt{2\cdot n - 1} \cdot K } \cdot  \sum_{i=0}^1 \text{vol}(\mathcal{K}_i) 
\end{align}
\end{lemma}
\begin{proof}
Let $\rho_0(x) = r_0^2 - \|x - C_0\|^2$, then $\frac{\partial \rho_0}{\partial x} = -2\cdot (x - C_0)^T$ hence $\frac{\partial \rho_0}{\partial x}(c + u \cdot t) \cdot u $ becomes for $c = C_0$
\begin{align}
 -2\cdot (c + u\cdot t - C_0)^T \cdot u \biggr|_{c = C_0} = -\|u\|^2 \cdot t = -t.
\end{align}

 Let $1 \leq T_1 \leq t \leq T_2$, then for $\rho_0(c + t \cdot u) \geq 0$ we have
\begin{align}\label{E23d}
&\int_{T_1}^{T_2} dt \cdot \left| H(\rho_0(c + u\cdot t)) - H_K(\rho_0(c + u \cdot t)) \right| = \nonumber \\
& = \frac{1}{\sqrt{\pi}}\cdot \int_{T_1}^{T_2} dt \cdot \rho_0(c + u\cdot t)  \int_{K}^{\infty}dy\cdot  e^{-(\rho_0(c + u \cdot t) \cdot y)^2}  \nonumber \\
& \leq \frac{1}{\sqrt{\pi}} \cdot \int_{T_1}^{T_2} dt\cdot t \cdot  \rho_0(c + u\cdot t)  \int_{K}^{\infty}dy\cdot  e^{-(\rho_0(c + u \cdot t) \cdot y)^2} \nonumber \\
& = \frac{1}{\sqrt{\pi}} \cdot \int_{K}^{\infty} dy \cdot \frac{e^{-\left(y \cdot \left(r^2 - \|u \cdot t \|^2 \right) \right)^2}}{2\cdot y^2} \biggr|_{t = T_1}^{T_2} \nonumber \\ & = \frac{1}{\sqrt{\pi}} \cdot \int_{K}^{\infty} dy\cdot  \frac{e^{-y^2 \cdot \left( r^2 - T_2^2\right)^2} - e^{-y^2 \cdot \left( r^2 - T_1^2 \right)^2}}{2\cdot y^2} 
 \leq \frac{1}{\sqrt{\pi}} \cdot \int_{K}^{\infty} \frac{1}{2\cdot y^2} \cdot dy \nonumber \\
& = \frac{1}{\sqrt{\pi}} \cdot \frac{1}{2 \cdot K} 
\end{align} It can be checked that the same result is obtained if $\rho_0(c + t \cdot u) \leq 0$. 

Consider the well known spherical coordinates change in variable:
\begin{align}
\begin{bmatrix} x_1 \\ x_2\\  \vdots \\ x_{n-1}\\ x_n
\end{bmatrix} = C_0 + \begin{bmatrix} r \cdot \cos(\phi_1) \\  r \cdot \sin(\phi_1) \cdot \cos(\phi_2) \\ \vdots \\ r \cdot \sin(\phi_1) \cdot \hdots \cdot \sin(\phi_{n-2}) \cdot \cos(\phi_{n-1})\\r \cdot \sin(\phi_1) \cdot \hdots \cdot \sin(\phi_{n-2})\cdot \sin(\phi_{n-1})\end{bmatrix} = C_0 + r \cdot u(\phi_1, \hdots, \phi_{n-1}).
\end{align} Denote $\Phi = \begin{bmatrix} \phi_1 &\hdots &\phi_{n-1} \end{bmatrix}^T$. In (\ref{E17d}) it is obtained:

\begin{align}\label{E25d}
&  \sqrt{\pi} \cdot \int_{0}^1 dx_1 \hdots \int_{0}^1 dx_n \cdot \left| H(\rho_0(x))  - H_K(\rho_0(x)) \right| = \nonumber \\
& = \int_{\phi_{1,a}}^{\phi_{1,b}} d\phi_1 \hdots \int_{\phi_{n-1,a}}^{\phi_{n-1,b}} d\phi_{n-1} \int_{r_a (\phi_1, \hdots, \phi_{n-1})}^{r_b(\phi_1, \hdots, \phi_{n-1})} dr \cdot r^{n-1}  \cdot \sin^{n-2}(\phi_1) \cdot \hdots \cdot \sin^{1}(\phi_{n-2}) \cdot \nonumber \\
& \cdot  \int_{K}^{\infty} dy \cdot |\rho_0(C_0 + r \cdot u(\phi_1, \hdots, \phi_{n-1}))| \cdot e^{-\left( \rho_0( C_0 + r \cdot u(\phi_1, \hdots, \phi_{n-1}) ) \cdot y \right)^2}
\end{align} where $\phi_{1,a}, \phi_{1,b}, \hdots, \phi_{n-1,a}, \phi_{n-1,b}, r_{a}(\Phi), r_b(\Phi)$ are in such a way that the integration is done over the unit hypercube. As such, the condition (\ref{E22f}) assures $r_a(\Phi) \geq 1$. 

 Using the C-B-S inequality, for the last two integral signs in (\ref{E25d}), one gets: 

\begin{align}
& \int_{r_a (\Phi)}^{r_b(\Phi)} dr \cdot r^{n-1} \cdot  \int_{K}^{\infty} dy \cdot |\rho_0(C_0 + r \cdot u(\Phi))| \cdot e^{-\left( y \cdot \rho_0( C_0 + r \cdot u(\Phi) )  \right)^2} \leq \nonumber \\
& \leq \sqrt{ \int_{r_a}^{r_b} dr \cdot r^{2\cdot (n-1)} \cdot \int_{r_a}^{r_b} dr \cdot \rho_0(C_0 + r \cdot u(\Phi))^2 \cdot \left(  \int_{K}^{\infty} dy  \cdot e^{-\left( y \cdot \rho_0( C_0 + r \cdot u(\Phi) )  \right)^2} \right)^2} \nonumber \\
& = \sqrt{ \int_{r_a}^{r_b} dr \cdot r^{2\cdot (n-1)} } \cdot \sqrt{ \int_{r_a}^{r_b} dr \cdot \rho_0(C_0 + r \cdot u(\Phi))^2 \cdot \left(  \int_{K}^{\infty} dy  \cdot e^{-\left( y \cdot \rho_0( C_0 + r \cdot u(\Phi) )  \right)^2} \right)^2}
\end{align} Since, one has 
\begin{align}
& \left| \int_{K}^{\infty} dy \cdot e^{-(y  \cdot \rho_0(C_0 + r \cdot u(\Phi)))^2} \right| = \left| \int_{K \cdot \rho_0(C_0 + r \cdot u(\Phi))}^{\infty} dz \cdot e^{-z^2} \cdot \frac{1}{\rho_0(C_0 + r \cdot u(\Phi))} \right| \nonumber \\
& \leq \frac{1}{ |\rho_0(C_0 + r \cdot u(\Phi)) | } \cdot \int_{-\infty}^{\infty} e^{-z^2}\cdot dz = \frac{\sqrt{\pi}}{|\rho_0(C_0 + r \cdot u(\Phi))|}
\end{align} it is obtained:
\begin{align}\label{E28d}
 &\int_{r_a}^{r_b} dr \cdot \rho_0(C_0 + r \cdot u(\Phi))^2 \cdot \left(  \int_{K}^{\infty} dy  \cdot e^{-\left( y \cdot \rho_0( C_0 + r \cdot u(\Phi) )  \right)^2} \right)^2 \leq \nonumber \\
&\leq \sqrt{\pi}\cdot \int_{r_a}^{r_b} dr \cdot |\rho_0(C_0 + r \cdot u(\Phi))| \int_{K}^{\infty} dy\cdot e^{-(y \cdot \rho_0(C_0 + r \cdot u(\Phi)))} \nonumber \\
& \leq^{\text{see }(\ref{E23d})} \frac{\sqrt{\pi}}{2 \cdot K}
\end{align} because $r_a \geq 1$. Furthermore, since $1 \leq r_a \leq r_b$ 
\begin{align}\label{E29d}
\sqrt{\int_{r_a}^{r_b} r^{2 \cdot (n-1)} \cdot dr } &= \sqrt{\frac{r^{2\cdot n - 1}}{2\cdot n - 1} \biggr|_{r_a}^{r_b}} = \sqrt{\frac{r_b^{2\cdot n - 1} - r_a^{2 \cdot n - 1}}{2\cdot n - 1}} \nonumber \\
& \leq \frac{r_b^n }{ \sqrt{r_b}\cdot  \sqrt{2\cdot n - 1}} \leq \frac{n}{\sqrt{2\cdot n - 1}} \int_{0}^{r_b} r^{n-1}\cdot dr
\end{align}

Replacing (\ref{E28d}) and (\ref{E29d}) in (\ref{E25d}) we get

\begin{align}
\int_{0}^1 dx_1 \hdots \int_{0}^1 dx_n \cdot \left| H(\rho_0(x))  - H_K(\rho_0(x)) \right| \leq \frac{n}{2\cdot K \cdot \sqrt{2 \cdot n - 1} } \cdot \text{vol}\left( \mathcal{K}\right)
\end{align} where $\mathcal{K} = \text{conv} \left( C_0, \mathcal{U}\right)$ is the cone formed by the convex hull of the unit hypercube $\mathcal{U}$ and the point $C_0$. Indeed,
\begin{align}
&\text{vol}\left( \mathcal{K}\right) = \nonumber \\
& = \int_{\phi_{1,a}}^{\phi_{1,b}} d\phi_1 \hdots \int_{\phi_{n-1,a}}^{\phi_{n-1,b}} d\phi_{n-1} \int_{0}^{r_b(\phi_1, \hdots, \phi_{n-1})} dr \cdot r^{n-1}  \cdot \sin^{n-2}(\phi_1) \cdot \hdots \cdot \sin^{1}(\phi_{n-2}) 
\end{align}

\end{proof}

Finally, we give the following remark regarding the volume of $\mathcal{K}_i$ for $i \in \{0,1\}$. 

\begin{remark}[\textbf{Upper bound on $\text{vol}(\mathcal{K}_i)$}] \label{R1} 
 For $i \in \{0,1\}$ we claim that the volume of $\mathcal{K}_i$ is bounded above by $\sqrt{n} \cdot \left( \frac{\sqrt{n}}{2} + \left\|C_i - \frac{1}{2} \cdot 1_{n \times 1} \right\| \right) $.

 Indeed, as shall be motivated below, the largest area of the projection of the unit hypercube $\mathcal{U}$ on a hyper-plane is $\sqrt{n}$. Let us define the hyperplane
\begin{align}
\mathcal{H}_i = \left\{x \in \mathbb{R}^n | \left( C_i - \frac{1}{2}\cdot 1_{n\times 1}\right)^T \cdot \left( x - \frac{1}{2}\cdot 1_{n\times 1} \right) = 0 \right\}
\end{align} 

Our reasoning is that $\mathcal{K}_i$ is included in the prism with the base formed by the projection of $\mathcal{U}$ on $\mathcal{H}_i$ and height given by $\frac{\sqrt{n}}{2} + \|C_i - \frac{1}{2} \cdot 1_{n \times 1}\|$. Let $\mathcal{P}$ denote this prism. The volume of this prism being the area of its base times its height, hence at most $\sqrt{n} \cdot \left( \frac{\sqrt{n}}{2} + \|C_i - \frac{1}{2} \cdot 1_{n \times 1}\| \right) $, is therefore an upper bound for $\text{vol}(\mathcal{K}_i)$.

According to \cite{McMullen} the volume of the orhogonal projection of the unit hypercube on a $k < n$ dimensional space is equal to the volume of the orthogonal projection of the unit hypercube on the $(n-k)$ dimensional orthogonal space. As a consequence, the maximum area of the projection on a $n-1$ space is equal to the maximum projection on a $1$ dimensional space, i.e. $\sqrt{n}$.  This argument was first seen by the author in a Mathematics Stack Exchange answer of the user Emanuele Paolini. 
%
%
\end{remark}

From Lemma \ref{L1} and Remark \ref{R1} the following corollary is given:
\begin{corollary}
Let $C_0, C_1 \in \mathbb{R}^n$, $\mathcal{L} = \mathcal{U} \cap \bar{\mathcal{B}}(C_0,r_0) \cap \bar{\mathcal{B}}(C_1, r_1)$ with 
\begin{align}
\min_{x \in \mathcal{U}} \|x - C_i\| \geq 1 \hspace{0.5cm} \forall i \in \{0,1\}
\end{align} then 
\begin{align}
| \text{vol}(\mathcal{L}) - T(K) | \leq \frac{n}{\sqrt{2\cdot n - 1} \cdot K } \cdot \sum_{i = 0}^1 \sqrt{n} \cdot \left( \frac{\sqrt{n}}{2} + \|C_i - \frac{1}{2} \cdot 1_{n \times 1}\| \right)
\end{align}
\end{corollary}

It is left for future work the obvious generalization of the above results for the volume of a finite (instead of just two treated above) intersection of balls with the unit hypercube. 

\subsection{Numerical Evaluation of $T(K)$ for a fixed $K$} \label{NE}

In the above section we established (for the hypercube clipped by balls) that the volume can be approximated with $T(K)$ with vanishing error as $K$ grows large. In this section we focus on computing $T(K)$. One rewrites $T(K)$ as follows:

\begin{align}
T(K) =& \int_{0}^{1} dx_1 \hdots \int_{0}^{1} dx_n \cdot  \left( \frac{1}{2} + \frac{\rho_0(x)}{\sqrt{\pi}} \cdot \int_{0}^{K} e^{-(y \cdot \rho_0(x))^2}\cdot dy  \right) \cdot \nonumber \\
& \cdot  \left( \frac{1}{2} + \frac{\rho_1(x)}{\sqrt{\pi}} \cdot \int_{0}^{K} e^{-(y \cdot \rho_1(x))^2}\cdot dy  \right) = I_1 + I_2 + I_3 + I_4
\end{align} a sum of four integrals as follows
\begin{align}
I_1(K) = \frac{1}{4} \cdot  \int_{0}^{1} dx_1 \hdots \int_{0}^{1} dx_n = \frac{1}{4}
\end{align} 

\begin{align}\label{E19}
&I_2(K) = \frac{1}{2\cdot \sqrt{\pi}} \cdot  \int_{0}^{1} dx_1 \hdots \int_{0}^{1} dx_n \cdot \rho_1(x) \cdot  \int_{0}^{K} e^{-y^2 \cdot \rho_1^2(x)} dy \nonumber \\ 
& = \frac{1}{2\cdot \sqrt{\pi}} \cdot  \int_{0}^{K_1} dy  \int_{0}^{1} dx_1 \hdots \int_{0}^{1} dx_n \cdot \left( d - \sum_{i=1}^n c_i(x_i)\right) \cdot e^{-y^2 \cdot ( d - \sum_{i=1}^n c_i(x_i))^2} \nonumber \\
&= \frac{d}{2\cdot \sqrt{\pi}} \cdot \int_{0}^{K_1} dy  \int_{0}^{1} dx_1 \hdots \int_{0}^{1} dx_n  \cdot e^{-y^2 \cdot ( d - \sum_{i=1}^n c_i(x_i))^2} -   \nonumber \\
& \hspace{0.5cm} - \frac{1}{2\cdot \sqrt{\pi}} \cdot \sum_{i=1}^n  \int_{0}^{K_1} dy  \int_{0}^{1} dx_1 \hdots \int_{0}^{1} dx_n \cdot   c_i(x_i) \cdot  e^{-y^2 \cdot ( d - \sum_{i=1}^n c_i(x_i))^2} 
\end{align} We assume w.l.o.g. that $d = 0$ since being a constant, it can be distributed to the functions $c_i(\cdot)$. Indeed, for example:
\begin{align}
 e^{-y^2 \cdot ( d - \sum_{i=1}^n c_i(x_i))^2} & = e^{- y^2 \cdot  \left( \sum_{i=1}^n \left( -c_i(x_i) + \frac{d}{n} \right) \right)^2} \nonumber \\
& = e^{-y^2 \cdot d^2} \cdot  e^{-y^2 \cdot \left( \sum_{i=1}^n \bar{c}_i(x_i)\right)^2}
\end{align} where $\bar{c}_i(x_i) = -c_i(x_i) + \frac{d}{n}$. 
 One gets 

\begin{align}\label{E21}
I_{2}(K) = \int_{0}^{K} dy \cdot \sum_{i}^{\text{poly}(n)} \alpha_i \cdot J_i(y)
\end{align} with $\alpha_i \in \mathbb{R}$ and $J_i(y)$ is an integral of the type
\begin{align}\label{E22}
{}^1J(y) = \int_{0}^1dx_1 \hdots \int_{0}^1 dx_n \cdot \prod_{i=1}^n f_i(x_i) \cdot e^{-y^2 \cdot \left( \sum_{i=1}^n g_i(x_i)\right)^2}
\end{align} where $f_i(\cdot), g_i(\cdot)$ are real univariate polynomials. 

Since $I_3(K)$ is similar to $I_2(K)$ we proceed to $I_4(K)$. 

\begin{align}\label{E23}
&I_4 = \frac{1}{\pi} \cdot   \int_{0}^{K} dy  \int_{0}^{K} dz \int_{0}^{1} dx_1 \hdots \int_{0}^{1} dx_n \cdot \rho_0(x) \cdot \rho_1(x) \cdot  e^{-y^2 \cdot \rho_0^2(x)} \cdot e^{-z^2 \cdot \rho_1^2(x)}  \nonumber \\ 
& \vdots \nonumber \\
& = \int_{0}^{K} dy  \int_{0}^{K} dz \cdot \sum_{i}^{\text{poly}(n)}\cdot \sum_{j}^{\text{poly}(n)} M_{i,j}(y,z)
\end{align} where $M_{i,j}(i,z)$ for a fixed $y,z$ is an integral of the following type:

\begin{align}
{}^2J(y) = \int_{0}^1dx_1 \hdots \int_{0}^1 dx_n \cdot \prod_{i=1}^n f_i(x_i) \cdot e^{-y^2 \cdot \left( \sum_{i=1}^n g_i(x_i)\right)^2} \cdot e^{-z^2 \cdot \left( \sum_{i=1}^n h_i(x_i)\right)^2}
\end{align} where $f_i(\cdot), g_i(\cdot), h_i(\cdot)$ are real univariate polynomials.

\subsection{Evaluating Integrals of the Type ${}^1J(y)$} \label{EI1}

Consider the integral:

\begin{align}\label{E25}
{}^1J(y) = \int_{0}^1dx_1 \hdots \int_{0}^1 dx_n \cdot \prod_{i=1}^n f_i(x_i) \cdot e^{-y^2 \cdot \left( \sum_{i=1}^n g_i(x_i)\right)^2}
\end{align} where $f_i(\cdot), g_i(\cdot)$ are real univariate polynomials with $|y| \leq K$ and $\left| \sum_{i=1}^n g_i(x_i) \right| \leq \text{poly}(n)$. 

In the following we present a method to numerically approximate  ${}^1J(y)$. For this we begin with the Taylor series of the exponential function:

\begin{align}\label{E26}
e^t = \sum_{k=0}^{\infty} \frac{1}{k!} \cdot t^k
\end{align} The convergence radius of this series is infinite, hence we obtain:
\begin{align}\label{E27}
e^{-y^2 \cdot \left(\sum_{i=1}^n g_i(x_i) \right)^2} = \lim_{p \to \infty}  \sum_{k=0}^{p} \frac{(-y^2)^k}{k!} \cdot \left( \sum_{q=1}^n g_i(x_i)\right)^{2 \cdot k}.
\end{align} As such we compute an integral of the type (\ref{E22}) as follows

\begin{align}\label{E42b}
&{}^1J(y) = \lim_{p \to \infty}  \sum_{k=0}^{p} \frac{(-y^2)^k}{k!}\int_{0}^{1} dx_1 \hdots \int_{0}^{1} dx_n \prod_{i=1}^n f_i(x_i) \cdot \left( \sum_{q=1}^n g_i(x_i)\right)^{2\cdot k} \nonumber \\
& = \lim_{p \to \infty}  \sum_{k=0}^{p} \frac{(-y^2)^k}{k!} \cdot J_k
\end{align} where 
\begin{align}
J_k = \int_{0}^{1} dx_1 \hdots \int_{0}^{1} dx_n \prod_{i=1}^n f_i(x_i) \cdot \left( \sum_{q=1}^n g_q(x_q)\right)^{2\cdot k}
\end{align} 

 Now, we compute each integral $J_k$ as follows. For simplicity of presentation assume w.l.o.g. that $n$ is a power of two. 
It is obtained:

\begin{align}
& J_k = \int_{0}^{1} dx_1 \hdots \int_{0}^{1} dx_n \cdot \prod_{q=1}^n f_q(x_q) \cdot \left( \sum_{q=1}^n g_q(x_q)\right)^{2\cdot k} \nonumber \\
& = \int_{0}^{1} dx_1 \hdots \int_{0}^{1} dx_n \cdot \prod_{q=1}^n f_q(x_q) \cdot \left( \sum_{q=1}^{\frac{n}{2}} g_q(x_q) + \sum_{\frac{n}{2}+1}^n g_q(x_q) \right)^{2\cdot k} \nonumber \\
& =  \int_{0}^{1} dx_1 \hdots \int_{0}^{1} dx_n \cdot  \prod_{q=1}^n f_q(x_q) \cdot \sum_{i=0}^{2\cdot k} C_{2\cdot k}^{i} \cdot \left( \sum_{q=1}^{\frac{n}{2}} g_q(x_q)\right)^{2\cdot k - i}  \cdot \left(\sum_{\frac{n}{2}+1}^n g_q(x_q) \right)^{i} \nonumber \\
& = \sum_{i = 0}^{2\cdot k}  C_{2\cdot k}^{i} \cdot  \int_{0}^{1} dx_1 \hdots \int_{0}^{1} dx_n \cdot \prod_{q=1}^n f_q(x_q) \cdot \left( \sum_{q=1}^{\frac{n}{2}} g_q(x_q)\right)^{2\cdot k - i}  \cdot \left(\sum_{\frac{n}{2}+1}^n g_q(x_q) \right)^{i} \nonumber \\
& = \sum_{i=0}^{2\cdot k} C_{2\cdot k}^i \cdot M_i
\end{align} where

\begin{align}
M_i & = \int_{0}^{1} dx_1 \hdots \int_{0}^{1} dx_{\frac{n}{2}} \cdot  \prod_{q=1}^{\frac{n}{2}} f_q(x_q) \cdot \left( \sum_{q=1}^{\frac{n}{2}} g_q(x_q) \right)^{2\cdot k - i} \cdot \nonumber \\
& \cdot \int_{0}^{1} dx_{\frac{n}{2}+1} \hdots \int_{0}^{1} dx_{n} \cdot  \prod_{q=\frac{n}{2}+1}^{n} f_q(x_q) \cdot \left( \sum_{\frac{n}{2}+1}^n g_q(x_q) \right)^{i}  \nonumber \\
& = M_{i,a} \cdot M_{i,b}
\end{align} where $M_{i,a}$ and $M_{i,b}$ are two integrals of the same form to $J_k$ but each having only half the variables. We repeat the same procedure, with each of them, to split each in a sum of integrals with again half the variables. Overall, we have:
\begin{enumerate}
\item the computation of $J_k$ requires $2 \cdot k+1$ integrals $M_i$
\item each $M_i$ is a product of two integrals $M_{i,a}, M_{i,b}$ of the same form to $J_k$ but each has only half the variables.  
\item therefore $J_k$ splits into $(2\cdot k + 1) \cdot 2$ integrals similar to it, but having each half the number of variables.
\item each of these integrals is applied the same procedure to write it as at most a sum of $(2\cdot k + 1) \cdot 2$ integrals with half the variables, (this is an upper bound). 
\item the process will continue at most $\log(n)$ times, until the obtained integrals will have only one variable. This one variable integrals are easily evaluated numerically, or if $f_i(\cdot)$ are polynomials, then  the univariate integral can be evaluated exactly.  This is our case since $c_i(\cdot), a_i(\cdot)$ are assumed to be polynomials. 
\item Overall there will be at most $(4\cdot k + 2)^{\log(n)} $ one variable integrals to be computed. 
\end{enumerate}

As such it has been shown that $J_k$ can be numerically evaluated in quasi-polynomial time. We end the section with an analysis of the number of terms required in (\ref{E27}). Using Taylor's theorem with the Lagrange reminder, we get in (\ref{E26})
\begin{align}
e^t - \sum_{k=1}^{p} \frac{1}{k!}\cdot t^k &= \frac{1}{(p+1)!} \cdot \left( e^t\right)^{(p+1)} \biggr|_{t = \xi} \cdot  t^{p+1} \nonumber \\
& = \frac{e^{\xi}}{(p+1)!} \cdot  t^{p+1}
\end{align} for some $\xi$ between $t$ and $0$. In out case $t = -y^2 \cdot \left( \sum_{i=1}^n g_i(x_i) \right) $. Under the stated assumptions $ \left| \sum_{i=1}^n g_i(x_i) \right| \leq \text{poly}(n)$, and $y \leq K$ hence $|t| \leq \text{poly}(n) \cdot K^2 \leq N(n) \cdot K^2$ for some polynomial $N(\cdot)$. Take $p = N^2(n) \cdot K^4$ to obtain 
\begin{align}\label{E33}
e^t - \sum_{k=1}^{n^2} \frac{1}{k!}\cdot t^k \leq \frac{e^{N(n) \cdot K^2}}{((N \cdot K^2 )^2+1)!} \cdot  (N \cdot K^2)^{(N \cdot K^2)^2+1} \to^{ N \cdot K^2 \to \infty} 0
\end{align} therefore retaining $N^2(n) \cdot K^4$ terms in the Taylor series vanishes the error for large dimensions. 

Finally, after retaining $N(n)^2\cdot K^4$ terms in the Taylor series, it turm out that the number of one dimensional integrals to be evaluated is bounded above by:
\begin{align}\label{E34}
N^2(n) \cdot K^4 \cdot \left( 4 \cdot N^2(n)\cdot K^4 + 2\right)^{\log(n)} \in \mathcal{O}\left( \text{poly}(n, K)^{\log(n)}\right)
\end{align}  i.e. quasi-polynomial.

\subsection{Evaluating Integrals of the Type ${}^2J(y,z)$} \label{EI2}
Consider the integral:

\begin{align}\label{E77e}
{}^2J = \int_{0}^1 dx_1 \int_{0}^1 dx_2 \hdots \int_{0}^1 dx_n \cdot \prod_{q=1}^n f_i(x_i) \cdot e^{-y^2 \cdot \left( \sum_{k=1}^n g_k(x_k)\right)^2} \cdot  e^{-z^2 \cdot \left( \sum_{l=1}^n h_l(x_l)\right)^2}
\end{align}  where $f_i(\cdot), g_k(\cdot), h_l(\cdot)$ are real univariate polynomials.  

We apply here the same approach as to (\ref{E22}), the previous type. From (\ref{E27}) we get 
\begin{align}
{}^2J =& \lim_{p_0 \to \infty} \lim_{p_1 \to \infty} \sum_{i=0}^{p_0} \frac{(-b_0)^i}{i!} \sum_{j=0}^{p_1} \frac{(-b_1)^j}{j!} \int_{0}^1 dx_1 \int_{0}^1 dx_2 \hdots \int_{0}^1 dx_n \cdot \nonumber \\
& \cdot \prod_{q=1}^n f_i(x_i) \cdot \left( \sum_{k = 1}^n g_k(x_k)\right)^{2\cdot i} \cdot \left( \sum_{l=1}^n h_l(x_l) \right)^{2\cdot j}
\end{align} Let $I_{i,j}$ denote a term in the double sum above and assume w.l.o.g. for the ease of presentation that $n$ is a power of two:

\begin{align}
&I_{i,j} = \int_{0}^1 dx_1 \int_{0}^1 dx_2 \hdots \int_{0}^1 dx_n \cdot \prod_{q=1}^n f_i(x_i) \cdot \left( \sum_{k = 1}^n g_k(x_k)\right)^{2\cdot i} \cdot \left( \sum_{l=1}^n h_l(x_l) \right)^{2\cdot j} \nonumber \\
& =  \int_{0}^1 dx_1 \int_{0}^1 dx_2 \hdots \int_{0}^1 dx_n \cdot \prod_{q=1}^n f_i(x_i) \cdot \left( \sum_{k = 1}^{\frac{n}{2}} g_k(x_k) + \sum_{k = \frac{n}{2}+1}^{n}g_k(x_k)\right)^{2\cdot i}  \cdot \nonumber \\
& \cdot \left( \sum_{l=1}^{\frac{n}{2}} h_l(x_l) + \sum_{l = \frac{n}{2}+1}^n h_l(x_l) \right)^{2\cdot j} \nonumber \\
& = \sum_{m=0}^{2\cdot i} C_{2\cdot i}^m \cdot \sum_{r=0}^{2\cdot j} C_{2\cdot j}^r \cdot \int_{0}^1 dx_1 \int_{0}^1 dx_2 \hdots \int_{0}^1 dx_n \cdot \prod_{q=1}^n f_i(x_i)  \cdot  \nonumber \\
& \cdot \left( \sum_{k=1}^{\frac{n}{2}} g_k(x_k) \right)^m \cdot \left( \sum_{k=\frac{n}{2}+1}^n g_k(x_k) \right)^{2\cdot i - m} \cdot \left( \sum_{l=1}^{\frac{n}{2}} h_l(x_l) \right)^{r} \cdot \left( \sum_{l=\frac{n}{2}+1}^n h_l(x_l) \right)^{2\cdot j - r} \nonumber \\
& = I_{i,j,A} \cdot I_{i,j,B}  
\end{align} which is a double sum with term of the form
\begin{align}
&\int_{0}^1 dx_1 \int_{0}^1 dx_2 \hdots \int_{0}^1 dx_\frac{n}{2} \cdot \prod_{q=1}^{\frac{n}{2}} f_i(x_i) \cdot \left( \sum_{k=1}^{\frac{n}{2}} g_k(x_k) \right)^m \cdot \left( \sum_{l=1}^{\frac{n}{2}} h_l(x_l) \right)^{r} \cdot  \nonumber \\
& \int_{0}^1 dx_{\frac{n}{2}+1} \hdots \int_{0}^1 dx_n \cdot \prod_{q={\frac{n}{2}+1}}^n f_i(x_i) \cdot \left( \sum_{k=\frac{n}{2}+1}^n g_k(x_k) \right)^{2\cdot i - m} \cdot \left( \sum_{l=\frac{n}{2}+1}^n h_l(x_l) \right)^{2\cdot j - r}
\end{align} i.e. the product of two integrals of the type $I_{i,j}$. Assume $i,j \leq M$ then $I_{i,j}$ requires at most $(2\cdot M+1)^2 $ integrals of the type $I_{i,j,A}$ and $I_{i,j,B}$. Each of these has half the variables of $I_{i,j}$. The process can continue for at most $\log(n) $ times before each obtained integral has only one variable. Such an intgral is then numerically evaluated, or for our case, since is the integral of a polynomial function, it is evaluated exactly . In total, there will be at most $(2\cdot M + 1)^{2 \cdot \log(n)} $ such integrals. Overall, as studied in the previous subsection, letting $p_0, p_1 \in \mathcal{O}(\text{poly}(K, n))$ one takes $M = \text{poly}(K, n)$ hence the number of univariate integrals are in $\mathcal{O}(\text{poly}(K, n)^{\log(n)}) $. 

\subsection{Integral Approximation Error Analysis}\label{ErAn}

Using (\ref{E33}) we can have an estimate of the error given by approximating the integrand in (\ref{E39}) as follows. A similar reasoning is valid for (\ref{E40}) as well, but we proceed with (\ref{E39}) since it is easier to present. The integrand is basically an n-integral of an exponential function which is replaced by its truncated Taylor series. The difference between  the exponential and its truncated Taylor  series is given by (\ref{E33}), therefore the n-integral of the absolute value of that difference is an upper bound of the obtained error. We therefore are able to compute in quasi-polynomial time ${}^1\hat{J}(y)$ for any $y $ with 

\begin{align}\label{E39b}
\left| {}^1J(y) - {}^1\hat{J}(y) \right| &\leq  \int_{0}^1 dx_1 \hdots \int_{0}^1dx_n \cdot \left|  \frac{e^{N(n) \cdot K^2}}{((N \cdot K^2 )^2+1)!} \cdot  (N \cdot K^2)^{(N \cdot K^2)^2+1}  \right| \nonumber \\
& \leq \left|  \frac{e^{N(n) \cdot K^2}}{((N \cdot K^2 )^2+1)!} \cdot  (N \cdot K^2)^{(N \cdot K^2)^2+1}  \right| 
\end{align}  where $\left| \sum_{i=1}^n g_i(x_i)\right| \leq N(n) \in \mathcal{O}(n)$ since $\left| \sum_{i=1}^n g_i(x_i)\right|\leq \sum_{i=1}^n \left| g_i(x_i) \right| \leq n \cdot G$ where $G = \max_{i} \max_{x \in [0,1]}  | g_i(x) |$.

The integrals $I_2, I_3$ are obtained with an error of the form:
\begin{align}
\left| I_2 - \hat{I_2} \right| &= \left| \int_{0}^K {}^1J(y) \cdot dy - \int_{0}^K {}^1\hat{J}(y)\cdot dy \right|  \nonumber \\
& \leq \int_{0}^K \left| {}^1J(y) - {}^1\hat{J}(y)\right| \cdot dy \nonumber \\
& \leq \int_{0}^K \left|  \frac{e^{N(n) \cdot K^2}}{((N \cdot K^2 )^2+1)!} \cdot  (N \cdot K^2)^{(N \cdot K^2)^2+1}  \right|  \cdot dy \nonumber \\
& \leq K \cdot \left|  \frac{e^{N(n) \cdot K^2}}{((N \cdot K^2 )^2+1)!} \cdot  (N \cdot K^2)^{(N \cdot K^2)^2+1}  \right|  \to^{K \text{ or } n \to \infty} 0
\end{align} with, as shown above,  $N(n) \in \mathcal{O}(n)$

Finally in order to numerically obtain the value of $\int_{0}^K {}^1\hat{J}(y) \cdot dy$
 one can simply approximate the integral with a Riemann sum:

\begin{align}
\int_{0}^K {}^1\hat{J}(y) \cdot dy \approx \sum_{i = 0}^{\frac{K}{\epsilon}-1} {}^1\hat{J}(i \cdot \epsilon) \cdot \epsilon =: S(\hat{J})
\end{align} for some $\epsilon > 0$ the integration step. Briefly, we estimate the estimation error, i.e. $\int_{0}^K {}^1\hat{J}(y) \cdot dy - S({}^1\hat{J})$ as
\begin{align}
\left| \int_{0}^K {}^1\hat{J}(y) \cdot dy - S({}^1\hat{J})\right| \leq \sum_{i=0}^{\frac{K}{\epsilon} -1 }\left| \int_{i \cdot \epsilon}^{i \cdot \epsilon + \epsilon} {}^1\hat{J}(y) \cdot dy - {}^1\hat{J}(i\cdot \epsilon) \cdot \epsilon \right|.
\end{align} Applying the Mean Value Theorem (MVT), using (\ref{E39b}) then applying MVT again, we get
\begin{align}\label{E43b}
\left|\int_{i \cdot \epsilon}^{i \cdot \epsilon + \epsilon} {}^1\hat{J}(y) \cdot dy - {}^1\hat{J}(i\cdot \epsilon) \cdot \epsilon \right| &= \left|{}^1\hat{J}(\hat{y}) \cdot \epsilon - {}^1 \hat{J}(i\cdot \epsilon) \cdot \epsilon \right| \approx \left| {}^1J(\hat{y}) - {}^1J(i\cdot \epsilon) \right| \cdot \epsilon \nonumber \\
& \approx \left| \frac{d}{dy} {}^1J(y) \biggr|_{y = \hat{\hat{y}}} \cdot (\hat{y} - i\cdot \epsilon) \cdot \epsilon\right| \nonumber \\
& \leq \left| \frac{d}{dy} {}^1J(y) \biggr|_{y = \hat{\hat{y}}} \right|  \cdot \epsilon^2
\end{align} where $\hat{y} \in (i\cdot \epsilon, i\cdot \epsilon + \epsilon)$ and $\hat{\hat{y}} \in (i\cdot \epsilon, \hat{y})$.  Recall from (\ref{E22}) that 
\begin{align}
{}^1J(y) = \int_{0}^1dx_1 \hdots \int_{0}^1 dx_n \cdot \prod_{i=1}^n f_i(x_i) \cdot e^{-y^2 \cdot \left( \sum_{i=1}^n g_i(x_i)\right)^2}
\end{align}  hence 
\begin{align}
& \frac{d}{dy} {}^1J(y) = \nonumber \\
& = -2\cdot y \cdot \int_{0}^1dx_1 \hdots \int_{0}^1 dx_n \cdot \prod_{i=1}^n f_i(x_i) \cdot e^{-y^2 \cdot \left( \sum_{i=1}^n g_i(x_i)\right)^2} \cdot  \left( \sum_{i=1}^n g_i(x_i)\right)^2
\end{align}  However, as seen  in (\ref{E19}) for our case, the term $\prod_{i=1}^n f_i(x_i) = c_i(x_i)$, i.e. only a finite number of $f_i(\cdot) \neq 1$. As such we have $\prod_{i=1}^n f_i(x_i) \in \mathcal{O}(1)$. Furthermore $\left| \left( \sum_{i=1}^n g_i(x_i)\right)^2\right| \leq \left( \sum_{i=1}^n |g_i(x_i)|\right)^2 \in \mathcal{O}(n^2)$ hence

\begin{align}
\left| \frac{d}{dy} {}^1J(y)\right| &\leq 2 \cdot K \cdot \int_{0}^1 dx_1 \hdots \int_{0}^1 dx_n \cdot \left|\prod_{i=1}^n f_i(x_i) \right| \cdot \left| \left( \sum_{i=1}^n g_i(x_i)\right)^2\right|  \nonumber \\
& \leq 2 \cdot K \cdot n^2 \cdot \mathcal{A} \in \mathcal{O}(K \cdot \text{poly}(n))
\end{align} where $\mathcal{A} \geq \left|\prod_{i=1}^n f_i(x_i) \right| \cdot |g_i(x_i)|^2 $ is a constant. As such (\ref{E43b}) becomes 
\begin{align}
\left|\int_{i \cdot \epsilon}^{i \cdot \epsilon + \epsilon} {}^1\hat{J}(y) \cdot dy - {}^1\hat{J}(i\cdot \epsilon) \cdot \epsilon \right| &\leq \epsilon^2 \cdot 2 \cdot K \cdot n^2 \cdot \mathcal{A} 
\end{align} hence (\ref{E42b}) becomes
\begin{align}
\left| \int_{0}^K {}^1\hat{J}(y) \cdot dy - S({}^1\hat{J})\right| &\leq  \epsilon \cdot 2 \cdot K \cdot n^2 \cdot \mathcal{A}  \cdot \sum_{i=0}^{\frac{K}{\epsilon} -1 } \epsilon \nonumber \\
& \leq \epsilon \cdot 2 \cdot \mathcal{A} \cdot K^2 \cdot n^2 = \epsilon \cdot \text{poly}(n) \cdot K^2
\end{align}

 This equation eventually decides the overall computational complexity of estimating the volume to a desired precision. As such, we need $\epsilon \in \mathcal{O} \left( \frac{\delta}{K^2 \cdot \text{poly}(n)}\right)$ to achieve an estimation of the integral with precision $\delta$, then $\mathcal{O}\left( \frac{K^3 \cdot \text{poly}(n)}{\delta}\right)$ evaluations of ${}^1J(\cdot)$. Each such evaluation has quasi-polynomial complexity as given by (\ref{E34})

 We close the paper, noting that derivatives with respect to $y$ can be computed (approximated) with the same complexity (and error) for the integrand ${}^1\hat{J}(y)$, hence better quadrature formulas may be proposed. The error analysis for the integral $I_4$ is similar to this one ($I_3, I_2$), and leads to similar results.

\color{black}

\section{Conclusion and future work}
In this paper we presented a method which allows the approximation of the volume of a clipped hypercube with one or two not necessarily convex sets described by equations of the type $\sum_{k=1}^n a_k(x_k) = b$. Denote the obtained se by $\mathcal{L}$.  Our approach basically approximates the volume with the values of a function $T(K)$. We show that as the argument of the function grows, the approximation improves, i.e. $\lim_{K \to \infty} T(K) = \text{vol}(\mathcal{L})$. 

 We then obtain the values of this function as a sum of four integrals $I_1(K), I_2(K), I_3(K), I_4(K)$.  Integral $I_1(K)$ can be computed exactly, while the integrals $I_{2}(K)$ and $I_3(K)$ are shown to be a sum with at most a polynomial number of terms, each term being an integral of the form 
\begin{align}\label{E39}
\int_{0}^K dy\cdot {}^1J(y)
\end{align} see (\ref{E21}),  while $I_4(K)$ is shown to be a double sum with at most a polynomial (in $n$, the space dimension) number of terms, each term being an integral of the form
\begin{align}\label{E40}
\int_{0}^K dy \int_{0}^K dz \cdot {}^2J(y,z)
\end{align} see (\ref{E23}).

We gave a quasi-polynomial algorithm to approximate ${}^1J(y)$ and ${}^2J(y,z)$. As such one can use a quadrature formula to compute numerically the value of $T(K)$ for any fixed $K$. Complexity analysis is also caried out and establishes an upper bound in $\mathcal{O}\left( \text{poly}\left( \frac{n}{\delta}\right) \right)$ for the number of evaluations needed of ${}^1J(y)$ and ${}^2J(y,z)$ such that the integrals (\ref{E39}) and (\ref{E40}) are approximated with a desired precision $\delta$, see Subsection \ref{ErAn}.

\textbf{We summarize the innovations of the paper in the following:}
\begin{enumerate}
\item In Subsection \ref{AH} we give an original approximation of the Heaviside step function in terms of an "incomplete" Gauss integral. Our approximation is useful when computing volumes because by changing the  integration order, the resulting integrand can be approximated using globally convergent McLaurin series.
\item In Subsection \ref{VAE} we show that it is sufficient to retain only a polynomial number of elements in a McLaurin series to attain a desired accuracy of a specific approximation. We prove this by showing that the series reminder in Lagrange form vanishes in certain conditions.
\item In Subsection \ref{EI1} and \ref{EI2} we give an original recurrent method which at each step obtains an $n-$integral as a sum of a polynomial number of products of $\frac{n}{2}-$integrals. This allows a quasi-polynomial overall complexity, since we end up with univariate integrals over some polynomial functions, which can be computed exactly.
\item In Subsection \ref{ErAn} we give an original upper bound on the quadrature error using Riemann sums and the magnitude of the derivative of the function to be integrated.  
\end{enumerate}

\end{document}